\documentclass{birkjour}

\usepackage{amssymb,amsmath,amsthm}

\usepackage{dsfont}
\usepackage{hyperref}

\newtheorem{prop}{Proposition}
\newtheorem{cor}[prop]{Corollary}

\newtheorem{lemma}[prop]{Lemma}

\theoremstyle{definition}
\newtheorem{defn}[prop]{Definition}
\newtheorem{example}[prop]{Example}

\title{A Note on Schanuel's Conjectures for Exponential Logarithmic Power Series Fields}

\author{Salma Kuhlmann}

\address{ Universit\"at Konstanz,\\ Fachbereich Mathematik und Statistik\\
78457 Konstanz, Germany.\\
Webpage: http://www.math.uni-konstanz.de/~kuhlmann/}
\email{salma.kuhlmann@uni-konstanz.de}

\author{Micka\"{e}l Matusinski}
\address{Institut de Math\'ematiques de Bordeaux,\\ Universit\'e Bordeaux 1,\\ 351 cours de la Lib\'eration\\
33405 Talence cedex, France.\\
Webpage: http://sites.google.com/site/mickaelmatusinski/}
\email{mickael.matusinski@u-bordeaux1.fr}

\author{Ahuva C. Shkop}
\address{Department of Mathematics,\\ Ben Gurion University of the Negev,\\
Be'er Sheva 84105, Israel\\
Webpage: http://www.math.bgu.ac.il/~shkopa/}
\email{shkopa@math.bgu.ac.il}

\begin{document}

\thanks{This joint work was inspired during the first author's visit to Ben Gurion University sponsored by the Institute for Advanced Studies
in mathematics at Ben Gurion University. The first author wishes to thank the institute for this opportunity.\\
The third author was supported by a postdoctoral fellowship funded by the Skirball Foundation via the Center for Advanced Studies in
Mathematics at Ben-Gurion University of the Negev. }

\begin{abstract}
In \cite{Ax}, J. Ax proved a transcendency theorem for certain differential fields of characteristic zero : the differential counterpart of
the still open Schanuel's conjecture about the exponential function over $\mathbb{C}$ \cite[page 30]{lang:intro-transc-numb}.
In this article, we derive from Ax's theorem transcendency results in the context of differential valued
exponential fields. In particular, we obtain results for exponential Hardy fields, Logarithmic-Exponential power
series fields and  Exponential-Logarithmic power series fields.

\end{abstract}

\maketitle

\subjclass{Primary 11J81; Secondary 12H05}
\keywords{Schanuel's conjectures, generalized series fields with derivation, exponential-logarithmic series fields with derivation}

In \cite{Ax}, J. Ax proves the following conjecture {\bf
(SD)} due to S. Schanuel, which is the differential counterpart of
the still open Schanuel's conjecture about the exponential function over $\mathbb{C}$ \cite[page 30]{lang:intro-transc-numb}.
Let $F$ be a field of characteristic $0$ and $D$ a derivation of $F$, i.e. $D(x+y) = D(x) + D(y) \mbox{ and
}D(xy)= xD(y)+yD(x)\>.$ We assume that the field of constants
$C$ contains $\mathbb{Q}$. Below td denotes the transcendence degree.

\medskip

\noindent{\bf (SD)} Let $y_1,...,y_n,z_1,...,z_n \in F^\times$ be such that
$Dy_k = \frac{Dz_k}{z_k} \mbox{ for } k= 1,...,n\>.$ If $\{Dy_k\>;\>
k= 1,\cdots,n\}$ is $\mathbb{Q}$-linearly independent, then
$$\mbox{td}_C C(y_1,...,y_n,z_1,...,z_n) \geq n+1.$$

\medskip

\noindent Consider the field of Laurent series $\mathbb{C}((t))$,
endowed with term by term derivation. The field of constants is
$\mathbb{C}$. Let $\mathbb{C}[[t]]$ denote the ring of formal power
series with complex coefficients in the variable $t$, i.e.
$\mathbb{C}[[t]]$ is the ring of Laurent series with nonnegative
exponents. It is a valuation ring (Definition 5). For any series
$y\in\mathbb{C}((t))$, let $y(0)$ denote its constant term. The
exponential map $\exp$ on $\mathbb{C}[[t]]$ is given by the Taylor
series expansion $\sum_{k\geq 0}\frac{y^k}{k!}$ and satisfies that
$Dy =
\frac{D\exp(y)}{\exp(y)}$.\\

\noindent A corollary to {\bf (SD)} which appears in Ax's paper is:

\medskip

\noindent {\bf Corollary B:} Let $y_i\in \mathbb{C}[[t]]$ such that
$y_i - y_i(0)$ are $\mathbb{Q}$-linearly independent,
$i=1,\cdots,n$. Then $\mbox{td}_\mathbb{C}
\mathbb{C}(y_1,...y_n,\exp(y_1),\cdots,\exp(y_n))\geq n+1$.

\medskip

\noindent We rephrase {\bf (SD)} as follows:

\noindent{\bf Theorem A:} Let $y_1,...,y_n,z_1,...,z_n \in F^\times$ be such
that $Dy_k = \frac{Dz_k}{z_k} \mbox{ for } k= 1,...,n\>.$ If
$\mbox{td}_CC(y_1,...,y_n, z_1,...,z_n) \leq n$, then $\sum_{i=1}^n
m_iy_i\in C$ for some $m_1,...,m_n \in \mathbb{Q}$ not all zero.

\medskip

\begin{defn}\label{defi:dvef} A \emph{differential valued exponential field} $K$ is a field of characteristic $0$, equipped with
a derivation $D:K \to K$, a valuation $v: K^{\times} \to G$ with value group $G$,
and an exponential map $\exp: K \to K^{\times}$ which satisfy the following:

\begin{itemize}
\item $\forall x \forall y (\exp(x+y) = \exp(x) \exp(y))$
\item $\forall x ( Dx= \frac{D\exp(x)}{\exp(x)} )$
\item The field of constants $C$ is isomorphic to the residue field of $v$.
\end{itemize}

\end{defn}

\noindent We denote the valuation ring by $\mathcal{O}_v$, the
maximal ideal by $\mathcal{M}_v$, and the residue field
$\bar{K}= \mathcal{O}_v /\mathcal{M}_v$. We thus require that
the field of constants $C\subseteq K$ is (isomorphic to) $\bar{K}$,
i.e. every element $y$ of $\mathcal{O}_v$ has a unique
representation $y= c + \epsilon$ where $c \in C$ and $\epsilon \in
\mathcal{M}_v$ (so $\mathcal{O}_v = C \oplus \mathcal{M}_v$). For
$y\in \mathcal{O}_v $ we write $\bar{y}$ for the residue of $y$
which is the above $c \in C$. In particular $\bar{c} = c$ for $c\in C$.
\medskip

\noindent In this note, we generalize Corollary B to {\it arbitrary}
elements of a differential valued exponential field $K$, that is, to
elements $y_1, \cdots, y_n$ {\it not} necessarily in the valuation
ring $\mathcal{O}_v$ (see Corollary \ref{main}). In particular we
apply Corollary \ref{main} to exponential-logarithmic series and
other examples (Examples \ref{exphardy} to \ref{EL}). Since $K$ is
not in general a field of series, we need to find an abstract
substitute for the ``constant term'' $y(0)$ of a series. Since the
additive group of $K$ is a $\mathbb{Q}$-vector space, and
$\mathcal{O}_v $ is a subspace, we choose and fix a vector space
complement $\bf{A}$ such that $K = {\bf A} \oplus C \oplus
\mathcal{M}_v$. For $y\in K^{\times}$ we define $\mbox{co}_{\bf
A}(y):=\overline{(y-a)}$ for the uniquely determined $a \in {\bf A}$
satisfying that $(y-a) \in \mathcal{O}_v$. Note that $\mbox{co}_{\bf
A}(y)=\bar{y}$ if $y \in \mathcal{O}_v$.

\bigskip\noindent
In this setting we observe:
\begin{lemma} \label{mainlemma}
 Let $y_1,\cdots, y_n \in K$ such that
$\sum_{i=1}^{n}m_iy_i \in C$ for some $m_i\in \mathbb{Q}$, then
$\sum_{i=1}^{n}m_i (y_i - \mbox{co}_{\bf A}(y_i)) = 0$.
\end{lemma}
\begin{proof}
Write $y_i = a_i + c_i + \epsilon _i$ with $a_i \in {\bf A}$, $c_i =
\mbox{co}_{\bf A}(y_i) \in C$ and $\epsilon _i \in \mathcal{M}_v$.
We compute:
\[\sum_{i=1}^{n}m_iy_i = \sum_{i=1}^{n}m_ia_i + \sum_{i=1}^{n}m_i\mbox{co}_{\bf A}(y_i)
+ \sum_{i=1}^{n}m_i\epsilon_i. \>\]
Since $\sum_{i=1}^{n}m_iy_i \in C$, it follows by the uniqueness of
the decomposition that $\sum_{i=1}^{n}m_ia_i =
\sum_{i=1}^{n}m_i\epsilon_i = 0$  as required.
\end{proof}
\begin{cor} \label{main} Let $y_1,...,y_n \in K$ and
suppose that $$y_1 - \mbox{co}_{\bf
A}(y_1),\cdots,y_n-\mbox{co}_{\bf A} (y_n)$$ are
$\mathbb{Q}$-linearly independent. Then
$$\mbox{td}_CC(y_1,...,y_n,\exp(y_1),...,\exp(y_n)) \geq n+1.$$
\end{cor}
\begin{proof}
Follows immediately from Theorem A and Lemma \ref{mainlemma}.
\end{proof}
\begin{example} \label{exphardy} Let $H$ be a \emph{Hardy field}, i.e. a set of germs at $+\infty$ of
real functions  which is a field and is closed under
differentiation. It carries canonically a valuation $v$
corresponding to the comparison relation between germs of real
function in this context. Moreover we suppose that $H$ carries an
exponential \cite[Definition p. 94]{Kuhlmann}, so $H$ is a
differential valued exponential field (see \cite{Rosenlicht}).
Assume that $f_1, \cdots, f_n \in H^{\times}$ such that $v(f_i)\geq
0$ i.e. $\lim_{x \rightarrow \infty} f_i(x)\in \mathbb{R}$. If
$$f_1 - \lim_{x \rightarrow \infty} f_1(x), \cdots, f_n - \lim_{x
\rightarrow \infty} f_n(x)$$ are $\mathbb{Q}$-linearly independent,
then
$$\mbox{td}_CC(f_1,...,f_n,\exp(f_1),...,\exp(f_n)) \geq n+1.$$
For example, take $H$ to be the field of logarithmic-exponential
functions defined by G.H. Hardy in \cite{hardy}.
\end{example}
\noindent We now consider fields of generalized series. By
Kaplansky's embedding theorem in \cite{kap}, generalized series
fields are universal domains for valued fields in the equal
characteristic case.
\begin{defn}
Let $k$ be a field of characteristic 0 and $G$ a totally ordered
Abelian group. A \emph{generalized series} (with coefficients $a_g$
in $k$ and exponents $g$ in $G$) is a formal sum $a=\sum_{g\in G}
a_g t^g$ with well-ordered support ${\rm Supp}\ a:=\{g\in G\ |\
a(g)\neq 0\}$. These series form a field, $K=k((G))$, under
component-wise sum and convolution product \cite{hahn}. We consider
the canonical valuation $v_{\mbox\rm{min}}$ on $K$ defined by
$v_{\mbox\rm{min}}(a):=\min({\rm Supp}\ a)$. The value group is $G$,
the residue field is $k$, and the valuation ring is $k[[G^{\geq
0}]]$ (the ring of series with support in the positive cone of $G$).
E.g. if $G=\mathbb{Z}$ then $\mathbb{C}((G))$ is the field of
Laurent series and $\mathbb{C}[[G^{\geq 0}]] = \mathbb{C}[[t]]$ the
ring of formal power series.
\end{defn}
\noindent For $k((G))$ a canonical complement to the valuation ring
is given by ${\bf A}:= k[[G^{< 0}]]$ (series with support in the
negative cone of $G$). Then for $y\in K^{\times}$ we have
$\mbox{co}_{\bf A}(y)=y(0)$ the constant term of $y$. More
generally, if $F$ is a truncation closed subfield of $K$ (i.e. $y
\in F$ implies that every initial segment of $y\in F$ as well), then
${\bf A}_F:= F\cap \> k[[G^{< 0}]]$ is a canonical complement and
$\mbox{co}_{\bf A_F}(y)=y(0)$ for $y\in F^{\times}.$

\bigskip \noindent
We note that the proof of Corollary \ref{main} does not require that
the domain of the exponential map is $K$. For instance, in the next
example the exponential is only defined on the valuation ring.
\begin{example}\label{generalisedseries}
Take $K=\mathbb{R}((G))$ endowed with a series derivation $D$
\cite[Definition 3.3]{KM2}. For $\epsilon\in \mathbb{R}((G^{>0}))\>$,
$\exp(\epsilon):=\sum_{k\geq 0} \frac{\epsilon ^k}{k!}$ is well-defined and satisfies
$\exp(x+y) = \exp(x) \exp(y)$. The exponential is defined for a
series $y(0) + \epsilon$ in the valuation ring $\mathbb{R}[[G^{\geq
0}]]$ via $\exp(y(0) + \epsilon)=\exp(y(0))\exp(\epsilon)$. Observe
that $D(\epsilon)= \frac{D(\exp(\epsilon))}{\exp(\epsilon)}$, since
$D$ is a series derivation \cite[Corollary 3.9]{VDH}. Thus Corollary
B holds for arbitrary value group $G$ instead of just
$G=\mathbb{Z}$.

\medskip \noindent
To make the example more explicit, we describe how to obtain series
derivations on $K$. Take $G$ to be a \emph{Hahn group} over some
totally ordered set $\Phi$, i.e. the lexicographical product of
$\Phi$ copies of $\mathbb{R}$ restricted to elements with
well-ordered support. (Recall that, by Hahn's embedding theorem in
\cite{hahn}, the Hahn groups are universal domains for ordered
abelian groups : see \cite[Section 2]{KM2}). For any $\phi\in\Phi$
we denote by $\mathds{1}_\phi$ the element of $G$ corresponding to 1
for $\phi$ and to 0 for the others elements of $\Phi$. Any $g\in G$
can be written $g=\sum_{\phi\in\Phi}g_\phi\mathds{1}_\phi$ with
$g_\phi\in\mathbb{R}$ and where the sum has well-ordered support. We
consider the following special cases.

\smallskip
\noindent \textbf{Case 1.} Suppose that $\Phi$ admits an order
preserving map $\sigma$ into itself which is a right-shift, i.e.
$\sigma(\phi)>\phi$ \cite[Section 5.1, Example 1]{KM2} (e.g. take
$\Phi$ a totally ordered abelian group and $\sigma$ translation by
fixed positive element). We set $D(1)=0$ and define $D$ on $\Phi$
by: \begin{itemize}
\item $D(t^{\mathds{1}_\phi}):=t^{\mathds{1}_\phi-\mathds{1}_{\sigma(\phi)}}$
 i.e.
$\frac{D(t^{\mathds{1}_\phi})}{t^{\mathds{1}_\phi}}=t^{-\mathds{1}_{\sigma(\phi)}}
\>.$
\end{itemize}
By \cite[Proposition 5.2]{KM2} we extend $D$ to a well-defined
derivation on $\mathbb{R}((G))$ via the following axioms of a series
derivation:
\begin{itemize}
\item \textbf{Strong Leibniz rule}: for $g=\sum_\phi g_\phi \mathds{1}_\phi$, set $D(t^g):=
t^g\sum_{\phi} g_\phi \frac{D(t^{\mathds{1}_\phi})}{t^{\mathds{1}_\phi}}=\sum_{\phi} g_\phi t^{g-\mathds{1}_{\sigma(\phi)}}$;
\item \textbf{Strong linearity}: for $a=\sum_{g} a_g t^g$, set $D(a):=\sum_{g}a_gD(t^g)$.
\end{itemize}

\smallskip
\noindent \textbf{Case 2.} Suppose that $\Phi$ is isomorphic (as an
ordered set) to a subset of $\mathbb{R}$ \cite[Section 5.1, Example
3]{KM2}. Note that $\Phi$ may not admit a right-shift (and indeed no non-trivial endomorphisms at all) \cite[Theorem 3]{dushnik-miller:similarity-tranfo}. In this case, we define $D$ as follows:
\begin{itemize}
\item  If $\Phi$ has a greatest element $\phi_M=\max \Phi$, take $f$ to be a fixed embedding of $\Phi$ into $\mathbb{R}$ and set for any $\phi\in\Phi$,
$\frac{D(t^{\mathds{1}_\phi})}{t^{\mathds{1}_\phi}}:=t^{f(\phi)\mathds{1}_{\phi_M}}$.
       \item If $\Phi$ has no greatest element fix an increasing sequence $\phi_n <\phi_{n+1}$ cofinal in $\Phi$. Then consider
       the corresponding partition of $\Phi$ made of sub-intervals of the form $S_n:=[\phi_n,\phi_{n+1})$. Take $f$ to be an embedding of $\Phi$ into $\mathbb{R}_{>0}$.
       For any $\phi\in\Phi$, there is $n$ such that $\phi\in S_n$, then set $\frac{D(t^{\mathds{1}_\phi})}{t^{\mathds{1}_\phi}}:=t^{f(\phi)\mathds{1}_{\phi_{n+1}}}$.
 \end{itemize}
 Once again, by \cite[Proposition 5.2]{KM2} we can extend the map $D$ to a well-defined derivation on the corresponding field
$\mathbb{R}((G))$ via the above axioms of a series derivation.
\end{example}

\begin{example} \label{LE}
The field of Logarithmic-Exponential (LE) series is a
differential valued exponential field \cite{MMV}. Moreover, as it is the
increasing union of power series $\mathbb{R}((G_n))$ it is a
truncation closed subfield of $\mathbb{R}((G_{\omega}))$ where
$G_{\omega}:= \cup G_n$. So Corollary \ref{main} applies to
LE-series $y_1, \cdots, y_n$ such that $y_1-y_1(0),...,y_n-y_n(0)$
are $\mathbb{Q}$-linearly independent. This generalizes Ax's result
Corollary B to Laurent series that are not necessarily in the
valuation ring.
\end{example}
\begin{example} \label{EL}
The fields of Exponential-Logarithmic series EL($\sigma$) are
differential valued exponential fields \cite[Section 5.3 (2) and
Theorem 6.2]{KM1}. They are truncation closed, so again Corollary
\ref{main} applies to $y_1, \cdots, y_n \in \mbox{ EL}(\sigma)$ such
that $y_1-y_1(0),...,y_n-y_n(0)$ are $\mathbb{Q}$-linearly
independent. More explicitely, if we consider $G$ as the Hahn group
over some totally ordered set $\Phi$ endowed with a right-shift
automorphism $\sigma$, the construction given in \cite[Section 5.3
(2)]{KM1} is as follows:
\begin{itemize}
\item for $\phi\in\Phi$, set $\log(t^{\mathds{1}_\phi}):=t^{-\mathds{1}_{\sigma(\phi)}}$
and $D(t^{\mathds{1}_\phi}):=t^{\mathds{1}_\phi-\sum_n \mathds{1}_{\sigma^n(\phi)}}$;
\item for $g=\sum_\phi g_\phi\mathds{1}_\phi$, set $\log(t^g):=\sum_\phi g_\phi t^{-\mathds{1}_{\sigma(\phi)}}$ and
$D(t^g):= \sum_\phi g_\phi t^{g-\sum_n
\mathds{1}_{\sigma^n(\phi)}}$;
\item for any $a:=\sum_g a_gt^g=a_{g_0} t^{g_0} (1+\epsilon)$ with $a_{g_0}>0$, set $\log(a)=\log(a_{g_0})+\log(t^{g_0} )+\sum_{n\geq 1} \epsilon^n/n$ and $D(a):=\sum_g a_gD(t^g)$
.
\end{itemize}
$D$ is a series derivation making EL($\sigma$) into a differential
valued exponential field.
\end{example}
\noindent {\bf Acknowledgment:} We thank J. Freitag and D. Marker
for providing us with the notes of a seminar given by D. Marker on
Ax's paper.

\end{document}